\documentclass[12pt]{amsart}

\usepackage{geometry}
\usepackage{amsmath}
\usepackage{amsfonts}
\usepackage{amssymb}
\usepackage{amscd}
\usepackage{amsthm}
\usepackage[matrix,arrow]{xy}

\theoremstyle{plain}
\newtheorem{theo}{Theorem}
\newtheorem{lemma}[theo]{Lemma}
\newtheorem{prop}[theo]{Proposition}

\newtheorem{cor}[theo]{Corollary}

\theoremstyle{remark}
\newtheorem{remark}[theo]{Remark}

\DeclareMathOperator{\FF}{\mathbb F}

\DeclareMathOperator{\Ab}{\mathbb A}

\def\Z{{\mathbb Z}}

\def\OO{{\mathcal O}}

\def\mm{{\mathfrak m}}
\def\Div{{\rm Div}}
\def\Pic{{\rm Pic}}
\def\Nm{{\rm Nm}}
\def\Hom{{\rm Hom}}
\def\Ext{{\rm Ext}}
\def\Fr{{\rm Fr}}

\title{Orthogonal to principal ideles}

\author{Sergey Gorchinskiy}
\address{Steklov Mathematical Institute of Russian Academy of Sciences, Moscow, Russia}
\address{National Research University Higher School of Economics, Moscow, Russia}

\email{gorchins@mi.ras.ru}

\date{}

\begin{document}

\maketitle

\begin{abstract}
We describe the orthogonal to the group of principal ideles with respect to the global tame symbol pairing on the group of ideles of a smooth projective algebraic curve over a field.
\end{abstract}

\section{Introduction}

Harmonic analysis on the group of adeles $\Ab_X$ of a smooth projective algebraic curve $X$ over a finite field plays a fundamental role in the Tate--Iwasawa method in the study of the zeta-function of $X$. A crucial fact here is that the field of rational functions $K$ on $X$, being a subgroup in $\Ab_X$, coincides with its own orthogonal $K^{\bot}$ with respect to a natural pairing on $\Ab_X$. Parshin~\cite{Par1},~\cite{Par2} found a version of the Tate--Iwasawa method based on the harmonic analysis on the group of ideles~$\Ab_X^*$. Thus a natural problem is to investigate a multiplicative analogue of the self-orthogonality of $K$, that is, to describe the orthogonal~$K^{*\,\bot}$ to the subgroup $K^*\subset \Ab^*_X$ with respect to the global tame symbol pairing
$$
(\cdot,\cdot)_X\;:\;\Ab_X^*\times \Ab_X^*\longrightarrow k^*\,,\qquad (f,g)_X=\mbox{$\prod\limits_{x\in X}\Nm_{k(x)/k}(f_x,g_x)_x$}\,,
$$
where $(\cdot,\cdot)_x$ are the local tame symbol pairings (or Hilbert symbols).

This question was treated recently by Mu\~{n}oz Porras, Navas Vicente, Pablos Romo, Plaza Mart\'in~\cite[Theor.\,5.5]{MP}. They considered the cases when $X$ is defined either over a finite field, or over the field of complex numbers. In the latter case, they were using analytic considerations with $\sigma$-functions and prime forms on Riemann surfaces.

In this note, using algebraic methods, we describe the orthogonal $K^{*\,\bot}$ when $X$ is defined over an arbitrary ground field $k$. The kernel $U$ of the global tame symbol pairing has an explicit description, see Lemma~\ref{lem:Uinfinite}, Lemma~\ref{lem:Ufinite}, and is clearly contained in $K^{*\,\bot}$. Thus, in order to describe $K^{*\,\bot}$ it is enough to describe the quotient $K^{*\,\bot}/(K^*\cdot U)$ and this is what we do, see~Theorem~\ref{theo:infinite}, Remark~\ref{rmk:maingener}, and Theorem~\ref{theo:finite}. In particular, when $k$ is algebraically closed, or, more generally, when the group $k^*$ is divisible, we show that there is an exact sequence
$$
0\longrightarrow  \Hom\big(\Pic^0(X),k^*\big)\longrightarrow K^{\,*\bot}/(K^*\cdot U)\longrightarrow\Pic^0(X)\longrightarrow 0\,.
$$

When $k$ is finite, we prove the equality ${K^{\,*\bot}=K^*\cdot U}$ (note that this differs from the description of $K^{*\,\bot}$ given in~\cite{MP}). Of course, one can show this fact using class field theory, but we take an opposite way. Namely, we prove this fact using non-degeneracy of the Weil pairing and a relation between the Weil pairing and the tame symbol, see Proposition~\ref{prop:Weiltame}. Then we deduce that there is a natural isomorphism between $C_X/(q-1)$, where $C_X=\Ab^*_X/K^*$, and the universal abelian $(q-1)$-torsion quotient of the Galois group of~$K$, see~Corollary~\ref{cor:CFT}. Note that this statement implies the second fundamental inequality for Kummer extensions in the function field case. We believe that we provide thus a more clear proof of this important step in the construction of the class field theory than the previously known sequence of tricks, see, e.g., the book of Artin and Tate~\cite[\S\,VI.2]{AT}.

The author is very grateful to D.\,V.\,Osipov and A.\,N.\,Parshin for many useful suggestions. The author is partially supported by Laboratory of Mirror Symmetry
NRU HSE, RF Government grant, ag. no. 14.641.31.0001

\section{Orthogonal to an isotropic subgroup}\label{sec:abstr}

Let~$A$ and~$N$ be abelian groups and let
$$
(\cdot,\cdot)\::\;A\times A\longrightarrow N
$$
be a bilinear pairing. For simplicity, we assume that the pairing is either symmetric or antisymmetric to avoid the difference between left and right orthogonals. However, Proposition~\ref{prop:key} below is valid for an arbitrary pairing as well.

For a subgroup $E\subset A$, denote by $E^{\bot}\subset A$ the orthogonal to $E$ in $A$ with respect to the pairing $(\cdot,\cdot)$. Let $B,C\subset A$ be subgroups which are isotropic with respect to the pairing~$(\cdot,\cdot)$, that is, ${B\subset B^{\bot}}$ and ${C\subset C^{\bot}}$. Put
$$
A'=(B\cap C)^{\bot}\subset A\,.
$$
Clearly, we have $B,C,B^{\bot},C^{\bot}\subset A'$. Our aim is to describe the quotient $B^{\bot}/B$ in terms of~$A'/(B+C)$.

Given abelian groups $H\subset G$ and an element $g\in G$, we usually denote by $[g]$ the class of $g$ in the quotient $G/H$ when it is clear from the context which subgroup $H$ in $G$ is considered. The pairing $(\cdot,\cdot)$ defines naturally the maps
$$
\alpha\;:\;C\longrightarrow \Hom(A/C,N)\,,\qquad \beta\;:\;B\cap C\longrightarrow \Hom(A/A',N)\,,
$$
both given by the formula $c\longmapsto \big([a] \mapsto (a,c)\big)$, where $a\in A$, $c\in C$.

\begin{prop}\label{prop:key}
Suppose that the following conditions are satisfied:
\begin{itemize}
\item[(i)]
the map $\alpha$ is an isomorphism;
\item[(ii)]
the natural map $\Hom\big(A/(B+C),N\big)\to \Hom\big(A'/(B+C),N\big)$ is surjective.
\end{itemize}
Then there is a decreasing filtration $B^{\bot}/B=F^0\supset F^1\supset F^2\supset F^3=0$ with the following adjoint quotients:
$$
F^0/F^1\simeq {\rm Im}\big(B^{\bot}\to A'/(B+C)\big)\,,\qquad F^1/F^2\simeq \Hom\big(A'/(B+C),N\big)\,,\qquad F^2\simeq {\rm Coker}(\beta)\,.
$$
\end{prop}
\begin{proof}
Since $B\subset A'$, there are embeddings
$$
B^{\bot}\supset A'^{\,\bot}\,,\qquad B^{\bot}\cap C\supset A'^{\,\bot}\cap C\supset B\cap C\,.
$$
Define the filtration as follows:
$$
F^1={\rm Im}(B^{\bot}\cap C\to B^{\bot}/B)\simeq (B^{\bot}\cap C)/(B\cap C)\,,
$$
$$
F^2={\rm Im}(A'^{\,\bot}\cap C\to B^{\bot}/B)\simeq (A'^{\,\bot}\cap C)/(B\cap C)\,.
$$
Let us describe the adjoint quotients. It follows from condition~(i) that there are isomorphisms
\begin{equation}\label{eq:auxilisom}
B^{\bot}\cap C\simeq \Hom\big(A/(B+C),N\big)\,,\qquad A'^{\,\bot}\cap C\simeq \Hom(A/A',N)\,.
\end{equation}
The second isomorphism in~\eqref{eq:auxilisom} implies that ${F^2\simeq {\rm Coker}(\beta)}$. Both isomorphisms in~\eqref{eq:auxilisom} together with condition~(ii) imply that the quotient ${F^1/F^2\simeq (B^{\bot}\cap C)/(A'^{\,\bot}\cap C)}$ is isomorphic to ${\Hom\big(A'/(B+C),N\big)}$.

It follows from the embedding $B\subset B^{\bot}$ that there is an equality
$$
B+(B^{\bot}\cap C)=B^{\bot}\cap (B+C)\,.
$$
Hence the quotient~$F^0/F^1$ is isomorphic to $B^{\bot}/\big(B^{\bot}\cap (B+C)\big)$, which is also isomorphic to image of the natural map $B^{\bot}\to A'/(B+C)$.
\end{proof}

Actually, condition~(i) of Proposition~\ref{prop:key} implies that the map $\beta$ is injective.

Now we give two corollaries of Proposition~\ref{prop:key}, which will be useful for the applications to the tame symbol pairing. Let us say that a bilinear pairing between abelian groups ${G\times H\to N}$ is unimodular if it induces isomorphisms ${G\simeq \Hom(H,N)}$ and ${H\simeq \Hom(G,N)}$.

\begin{cor}\label{cor:key}
Suppose that the following conditions are satisfied:
\begin{itemize}
\item[(i)]
the map $\alpha$ is an isomorphism;
\item[(ii)]
the natural pairing
$$
(\cdot,\cdot)\;:\;(B\cap C)\times A/(B+C)\longrightarrow N
$$
is unimodular.
\end{itemize}
Then there is an equality $B=B^{\bot}$.
\end{cor}
\begin{proof}
The isomorphism ${A/(B+C)\simeq \Hom(B\cap C,N)}$ implies that $A'=B+C$, that is, $A'/(B+C)=0$. Together with the isomorphism ${B\cap C\simeq \Hom\big(A/(B+C),N\big)}$ this implies that $\beta$ is an isomorphism. Thus by Proposition~\ref{prop:key}, we have $F^0/F^1=F^1/F^2=F^2=0$, whence $B=B^{\bot}$.
\end{proof}

Let us introduce more notation. We have an exact sequence
\begin{equation}\label{eq:secondexabs}
0\longrightarrow B/(B\cap C)\longrightarrow  A'/C\longrightarrow  A'/(B+C)\longrightarrow 0\,.
\end{equation}
Define a map
$$
\gamma\;:\; A'/(B+C)\longrightarrow\Ext^1\big(A'/(B+C),N\big)
$$
as follows. The class $[a]\in A'/(B+C)$ of an element $a\in A'$ is sent by $\gamma$ to the class of the extension of $A'/(B+C)$ by $N$ obtained as the push-out of extension~\eqref{eq:secondexabs} along the homomorphism
$$
\lambda_a\;:\; B/(B\cap C)\longrightarrow N\,,\qquad [b]\longmapsto (a,b)\,,
$$
where $b\in B$. The homomorphism $\lambda_a$ is well-defined, because $a\in A'$, so that $(a,B\cap C)=0$. The class of the obtained extension does not depend on the choice of $a\in A'$ with fixed $[a]\in A'/(B+C)$, because for any $b\in B$, the homomorphism $\lambda_b$ is trivial and for any $c\in C$, the homomorphism $\lambda_c$ extends to a well-defined homomorphism ${(c,\cdot)\colon A'/C\to N}$.

\begin{cor}\label{cor:split}
Suppose that the following conditions are satisfied:
\begin{itemize}
\item[(i)]
the map $\alpha$ is an isomorphism;
\item[(ii)]
the quotient $A/A'$ split out of $A$.
\end{itemize}
Then there is a decreasing filtration $B^{\bot}/B=F^0\supset F^1\supset F^2\supset F^3=0$ with the following adjoint quotients:
$$
F^0/F^1\simeq {\rm Ker}(\gamma)\,,\qquad F^1/F^2\simeq \Hom\big(A'/(B+C),N\big)\,,\qquad F^2\simeq {\rm Coker}(\beta)\,.
$$
\end{cor}
\begin{proof}
Let us show that the image of the natural map ${\zeta\colon B^{\bot}\to A'/(B+C)}$ coincides with the kernel of $\gamma$. It follows from the definition of $\gamma$ that ${\rm Im}(\zeta)\subset{\rm Ker}(\gamma)$.

Conversely, let $a\in A'$ be such that $\gamma[a]=0$. Then the push-out of exact sequence~\eqref{eq:secondexabs} along~$\lambda_a$ admits a splitting, or, equivalently, the map $\lambda_a\colon B/(B\cap C)\to N$ extends to a map ${\tilde{\lambda}\colon A'/C\to N}$. Condition~(ii) implies that $A/A'$ splits out of $A/C$ as well. Together with condition~(i) this implies that the natural map $C\to\Hom(A'/C,N)$ is surjective. Let $c\in C$ be sent to $\tilde{\lambda}$ under this map. Then $\lambda_{a-c}=0$, that is, $a-c\in B^{\bot}$. Since $[a]=[a-c]$ in $A'/(B+C)$, we see that $[a]$ is in the image of $\zeta$. This proves that ${\rm Im}(\zeta)={\rm Ker}(\gamma)$.

Now we conclude the proof applying Proposition~\ref{prop:key}.
\end{proof}

\section{Tame symbol pairing}

Let $X$ be a smooth projective curve over a field $k$ and let $K=k(X)$ be the field of rational functions on $X$. We suppose that $X$ is geometrically irreducible over $k$, that is, $k$ is algebraically closed in $K$, or, equivalently, $H^0(X,\OO_X)=k$.

Given a closed point $x\in X$, denote by $\widehat{\OO}_{X,x}$ the completion of the local ring $\OO_{X,x}$, by $\widehat{\mm}_x\subset \widehat{\OO}_{X,x}$ the maximal ideal, by $k(x)$ the residue field at the point~$x$, which is also the residue field of the local ring $\widehat{\OO}_{X,x}$, and denote by~$K_x$ the fraction field of $\widehat{\OO}_{X,x}$. Equivalently,~$K_x$ is the completion of the field $K$ with respect to the discrete valuation $\nu_x:K^*\to\Z$ defined by~$x$. Note that $k(x)$ is canonically a finite extension of~$k$. We have a local tame symbol pairing
$$
(\cdot,\cdot)_x\;:\;K_x^*\times K_x^*\longrightarrow k(x)^*\,,\qquad (f_x,g_x)_x= \big((-1)^{\nu_x(f_x)\nu_x(g_x)}f_x^{-\nu_x(g_x)}g_x^{\nu_x(f_x)}\big)(x)\,,
$$
where $f_x,g_x\in K_x^*$. The pairing $(\cdot,\cdot)_x$ is antisymmetric.

For each closed point $x\in X$, put $d_x=[k(x):k]$ and let $d$ be the greatest common divisor of the numbers $d_x$ over all~${x\in X}$:
\begin{equation}\label{eq:d}
d={\rm GCD}\,(\,d_x\mid x\in X)\,.
\end{equation}

Denote by $\Ab_X$ the ring of adeles of $X$ and by $\Ab^*_X$ the group of ideles of $X$, that is, the group of invertible elements in $\Ab_X$. We have a global tame symbol pairing
$$
(\cdot,\cdot)_X\;:\;\Ab_X^*\times \Ab_X^*\longrightarrow k^*\,,\qquad (f,g)_X=\mbox{$\prod\limits_{x\in X}\Nm_{k(x)/k}(f_x,g_x)_x$}\,,
$$
where $f=(f_x)_{x\in X},\,g=(g_x)_{x\in X}\in\Ab_X^*$. The pairing $(\cdot,\cdot)_X$ is antisymmetric. By the explicit formula for the local tame symbol pairing, the subgroup ${\mbox{$\prod\limits_{x\in X}\widehat\OO^*_{X,x}$}\subset \Ab^*_X}$ is isotropic. By Weil reciprocity law, the subgroup $K^*\subset \Ab^*_X$ is isotropic as well. Let $U\subset \Ab^*_X$ denote the kernel of the global tame symbol pairing.

Our aim is to describe the quotient $K^{*\,\bot}/(K^*\cdot U)$. For this, we will apply results from Section~\ref{sec:abstr} to
$$
A=\Ab^*_X/U\,,\qquad B=K^*/(K^*\cap U)\,,\qquad C={\mbox{$\prod\limits_{x\in X}\widehat\OO^*_{X,x}$}}\,/\Big(\,{\mbox{$\prod\limits_{x\in X}\widehat\OO^*_{X,x}$}}\cap U\Big)\,,\qquad N=k^*\,,
$$
and to the pairing on $A$ induced by the global tame symbol pairing $(\cdot,\cdot)_X$.

\begin{remark}\label{rmk:kerneltame}
Clearly, an idele $f=(f_x)_{x\in X}\in\Ab^*_X$ belongs to $U$ if and only if for any $x\in X$, we have $f_x\in U_x$, where $U_x\subset K^*_x$ denotes the kernel of the pairing
$$
K_x^*\times K_x^*\longrightarrow k^*\,,\qquad (f_x,g_x)\longmapsto \Nm_{k(x)/k}(f_x,g_x)_x\,.
$$
An explicit description of $U_x$ and $U$ depends on whether the field~$k$ is infinite or finite.
\end{remark}

\section{The case of an infinite ground field}

Assume that the field $k$ is infinite.

\begin{lemma}\label{lem:Uinfinite}
There are equalities
$$
U_x={\rm Ker}\big(\widehat{\OO}^*_{X,x}\to k(x)^*\to k^*\big)\,,\qquad U=\mbox{$\prod\limits_{x\in X}U_x$}\,,
$$
where $x\in X$ is any closed point, the map $\widehat{\OO}^*_{X,x}\to k(x)^*$ is the natural surjective homomorphism, and the map $k(x)^*\to k^*$ is the norm map $\Nm_{k(x)/k}$.
\end{lemma}
\begin{proof}
The formula for $U_x$ follows from the explicit description of the local tame symbol pairing (here we use that $k$ is infinite). The formula for $U$ follows from Remark~\ref{rmk:kerneltame}.
\end{proof}

Denote by $\Div(X)$ the group of divisors on $X$, by $\Div^0(X)\subset \Div(X)$ the subgroup of degree zero divisors, by $\Pic(X)$ the Picard group of $X$, and by $\Pic^0(X)\subset\Pic(X)$ the subgroup of classes of degree zero divisors. Define a map
$$
\theta\;:\;\Pic^0(X)\longrightarrow {\rm Ext}^1\big(\Pic^0(X),k^*\big)
$$
as follows. An element $\ell\in \Pic^0(X)$ is sent by $\theta$ to the class of the extension of $\Pic^0(X)$ by $k^*$ obtained as the restriction of the Poincar\'e biextension over $\Pic^0(X)\times\Pic^0(X)$ to~${\{\ell\}\times\Pic^0(X)}$.

\begin{theo}\label{theo:infinite}
Assume that the field $k$ is infinite.
\begin{itemize}
\item[(i)]
Suppose that for any finite extension of fields $k\subset l$, the norm map $\Nm_{l/k}\colon l^*\to k^*$ is surjective. Then there is a decreasing filtration
$$
K^{*\,\bot}/(K^*\cdot U)=F^0\supset F^1\supset F^2\supset F^3=0
$$
with the following adjoint quotients (see~\eqref{eq:d} for the definition of $d$):
$$
F^0/F^1\simeq {\rm Ker}(\theta)\subset \Pic^0(X)\,,\qquad F^1/F^2\simeq \Hom\big(\Pic^0(X),k^*\big)\,,\qquad F^2\simeq k^*/(k^*)^d\,.
$$
\item[(ii)]
Suppose that the group $k^*$ is divisible. Then there is an exact sequence
$$
1\longrightarrow \Hom\big(\Pic^0(X),k^*\big)\longrightarrow K^{*\,\bot}/(K^*\cdot U)\longrightarrow \Pic^0(X)\longrightarrow 0\,.
$$
\end{itemize}
\end{theo}

The condition in Theorem~\ref{theo:infinite}(i) holds, in particular, if $k$ is quasi-algebraically closed and of zero characteristic, see, e.g., Serre's book~\cite[Prop.\,X.10, Prop.\,X.11]{Ser}. The condition in Theorem~\ref{theo:infinite}(ii) holds, in particular, if $k$ is algebraically closed.

\begin{proof}[Proof of Theorem~\ref{theo:infinite}]
(i) Let us describe in our case the groups $A$, $B$, $C$, $B\cap C$, $A'$ and the maps~$\alpha$,~$\beta$,~$\gamma$ from Section~\ref{sec:abstr}. Lemma~\ref{lem:Uinfinite} together with the surjectivity property of the norm maps imply the equalities
$$
\mbox{$C=\prod\limits_{x\in X}\widehat{\OO}^*_{X,x}/U_x=\prod\limits_{x\in X}k^*$}\,,\qquad A/C=\Div(X)\,.
$$
It follows that $B\cap C$ is the image of the map
$$
k^*\longrightarrow\mbox{$\prod\limits_{x\in X}k^*$}\,,\qquad c\longmapsto (c^{d_x})_{x\in X}\,.
$$
Note that for all $f\in \Ab^*_X$ and $c\in k^*$, we have $(f,c)_X=c^{\deg(f)}$, where
$$
\deg\;:\;\Ab_X^*\longrightarrow \Z\,,\qquad (f_x)_{x\in X}\longmapsto \mbox{$\sum\limits_{x\in X}d_x\,\nu_x(f_x)$}\,,
$$
is the degree homomorphism. Therefore, $A'=(\Ab^*_X)^0/U$, where $(\Ab^*_X)^0$ is the kernel of the degree homomorphism.
We see that $A'/(B+C)=\Pic^0(X)$.

It is easy to see that the map $\alpha$ coincides with the natural isomorphism
$$
\mbox{$\prod\limits_{x\in X}k^*$}\stackrel{\sim}\longrightarrow\Hom\big(\Div(X),k^*\big)\,.
$$
So, condition~(i) of Corollary~\ref{cor:split} is satisfied.

Further, the isomorphism $\deg\colon A/A'\stackrel{\sim}\longrightarrow d\,\Z$ implies that $A/A'$ splits out of $A$, whence condition~(ii) of Corollary~\ref{cor:split} is satisfied as well. Also, we obtain that ${\rm Coker}(\beta)$ is the cokernel of the map
$$
k^*\longrightarrow B\cap C\longrightarrow\Hom(A/A',N)\simeq k^*\,, \qquad c\longmapsto c^d\,.
$$
Therefore, ${{\rm Coker}(\beta)=k^*/(k^*)^d}$.

Finally, we show that the maps $\gamma$ and $\theta$ coincide up to sign. It is proved in~\cite[Theor.\,3.1]{Gor} that the Poincar\'e biextension over $\Pic^0(X)\times\Pic^0(X)$ is isomorphic to the quotient of the trivial biextension $k^*\times(\Ab^*_X)^0\times(\Ab^*_X)^0$ by the following action of the group $\big({K^*\times \prod\limits_{x\in X}\widehat{\OO}_{X,x}^*}\big)^{\times 2}$:
$$
\big((\varphi,u),(\psi,v)\big)\;:\;(c,f,g)\longmapsto \big(c(f,\psi)_X(u,\psi)_X(u,g)_X,f\varphi u,g\psi v\big)\,.
$$
The proof is based on the fact that the global tame symbol pairing coincides with the commutator pairing for the central extension of $\Ab^*_X$ by $k^*$ constructed by Arbarello, Kac, de Concini~\cite{AKC}. Define a map
\begin{equation}\label{eq:divadeles}
{\rm div}\;:\;\Ab^*_X\longrightarrow\Div(X)\,,\qquad (f_x)_{x\in X}\longmapsto \mbox{$\sum\limits_{x\in X}\nu_x(f_x)\cdot x$}\,.
\end{equation}
We see that for any idele $f\in(\Ab^*_X)^0$, the class $\theta[{\rm div}(f)]$ is equal to the class of the extension of $\Pic^0(X)$ by $k^*$ given by the cokernel of the homomorphism
$$
K^*\longrightarrow k^*\times \Div^0(X)\,,\qquad \psi\longmapsto \big((f,\psi)_X,{\rm div}(\psi)\big)\,.
$$
The latter extension coincides up to sign with the extension obtained as the push-out along the map $\lambda_f$ of extension~\eqref{eq:secondexabs} in our case. This proves that the maps~$\gamma$ and~$\theta$ are equal up to sign.

Now we conclude the proof applying Corollary~\ref{cor:split}.

(ii) By the assumption, the group $k^*/(k^*)^d$ is trivial. Moreover, the group~$k^*$ is injective as a $\Z$-module, whence ${\Ext^1\big(\Pic^0(X),k^*\big)=0}$ and ${{\rm Ker}(\theta)=\Pic^0(X)}$. Also, clearly, the assumption in~(i) is satisfied. Hence we conclude the proof applying~(i).
\end{proof}

\begin{remark}\label{rmk:maingener}
One can prove a generalization of Proposition~\ref{prop:key} for arbitrary isotropic subgroups ${B,C\subset A}$ without assuming conditions (i) and (ii) therein. This implies the following generalization of Theorem~\ref{theo:infinite} for an arbitrary infinite ground field $k$. For each closed point $x\in X$, let $\Gamma_x\subset k^*$ denote the image of the norm map $k(x)^*\to k^*$. Define an injective map~$\iota$ by the formula
$$
\iota\;:\;k^*\longrightarrow \mbox{$\prod\limits_{x\in X}k^*$}\,,\qquad c\longmapsto(c^{d_x/d})\,.
$$
Then there is a filtration
$$
K^{*\,\bot}/(K^*\cdot U)=F^0\supset F^1\supset F^2\supset F^3=0
$$
with the following adjoint quotients:
$$
F^0/F^1\simeq{\rm Ker}\left[{\rm Ker(\theta)}\longrightarrow {\rm Coker}\Big( \Hom\big(\Pic^0(X),k^*\big)\to \big(\mbox{$\prod\limits_{x\in X}k^*\big)/\big(\iota(k^*)\cdot\prod\limits_{x\in X}\Gamma_x\big)$} \Big)\right]\,,
$$
$$
F^1/F^2\simeq \mbox{$\prod\limits_{x\in X}\Gamma_x$}/\big(\iota(k^*)\cap \mbox{$\prod\limits_{x\in X}\Gamma_x$}\big)\cap \Hom\big(\Pic^0(X),k^*\big)\subset {\big(\mbox{$\prod\limits_{x\in X}k^*$}\big)/\iota(k^*)}\,,
$$
$$
F^2\simeq\big(\iota(k^*)\cap\mbox{$\prod\limits_{x\in X}\Gamma_x$}\big)/\iota(k^*)^d\,,
$$
If $\Gamma_x=k^*$ for all $x\in X$, this specializes to Theorem~\ref{theo:infinite}(i).
\end{remark}

\section{The case of a finite ground field}

Assume that $k=\FF_q$ is a finite field. We will use the following results.

\begin{prop}\label{prop:AT}
For any finite Galois extension of fields $K\subset L$, almost all valuations of $K$ split completely in $L$ if and only if $K=L$.
\end{prop}

This holds more generally when $K$ is an arbitrary global field, see, e.g.,~\cite[Theor.\,V.2]{AT}. As is noticed in op.cit., this is a consequence of the first fundamental inequality in class field theory, which is essentially reduced to the Riemann--Roch theorem in the function field case.

\begin{cor}\label{cor:AT}
\hspace{0cm}
\begin{itemize}
\item[(i)]
The natural map $K^*/(K^*)^{q-1}\to \Ab_X^*/(\Ab_X^*)^{q-1}$ is injective.
\item[(ii)]
There is an equality $d=1$ (see~\eqref{eq:d} for the definition of $d$).
\end{itemize}
\end{cor}
\begin{proof}
(i) Given an element $\varphi\in K^*$, we have $\varphi\in(\Ab^*_X)^{q-1}$ if and only if any valuation of $K$ splits completely in $K(\varphi^{\frac{1}{q-1}})$. By Proposition~\ref{prop:AT}, the latter is equivalent to $\varphi\in (K^*)^{q-1}$.

(ii) This is proved, e.g., in~\cite[Theor.\,V.5]{AT}. Indeed, apply Proposition~\ref{prop:AT} to the Galois extension of fields $K\subset K\otimes_{\FF_q}\FF_{q^d}$.
\end{proof}

Denote by $\overline X$ the curve $X\times_k\bar{k}$ over $\bar k$ and denote by $\Fr\colon \overline X\to \overline X$ the $\bar k$-linear $q$-th Frobenius morphism. This is a purely inseperable finite morphism of degree $q$. We have a group homomorphism $\Fr_*\colon\Pic(\overline X)\to\Pic(\overline X)$.

Note that the subgroup $k^*\subset \bar k^*$ coincides with the group $\mu_{q-1}$ of degree $q-1$ roots of unity in~$\bar k^*$. Denote by
$$
(\cdot,\cdot)_{q-1}\;:\;\Pic^0(\overline X)_{q-1}\times \Pic^0(\overline X)_{q-1}\longrightarrow \mu_{q-1}=k^*
$$
the corresponding Weil pairing.

\begin{lemma}\label{lemma:nondegen}
There is a well-defined unimodular pairing
$$
\kappa\;:\;\Pic^0(X)_{q-1}\times \Pic^0(X)/(q-1)\longrightarrow k^*\,,\qquad \kappa(\ell,[m])=(\ell,\Fr_*(\widetilde m)-\widetilde m)_{q-1}\,,
$$
where for any $m\in \Pic^0(X)$, an element $\widetilde m\in \Pic^0(\overline X)$ is such that $(q-1)\widetilde m=m$.
\end{lemma}
\begin{proof}
This is a rather standard fact. Namely, the Weil pairing $(\cdot,\cdot)_{q-1}$ is unimodular and the map $\Fr_*$ from $\Pic^0(\overline{X})_{q-1}$ to itself is an isometry with respect to the Weil pairing. Consider the map $\Fr_*-1$ from ${\Pic^0(\overline{X})_{q-1}}$ to itself. The Weil pairing induces a pairing
$$
{\rm Ker}(\Fr_*-1)\times {\rm Coker}(\Fr_*-1)\longrightarrow k^*
$$
such that the natural map ${{\rm Ker}(\Fr_*-1)\to{\Hom\big({\rm Coker}(\Fr_*-1),k^*\big)}}$ is injective. Since the finite groups ${\rm Ker}(\Fr_*-1)$ and ${\rm Coker}(\Fr_*-1)$ have the same orders, the above map is in fact an isomorphism. Since both groups are $(q-1)$-torsion and $k^*$ is a cyclic group of order~$q-1$, we obtain that the groups ${\rm Ker}(\Fr_*-1)$ and ${\rm Coker}(\Fr_*-1)$ are Pontryagin dual to each other. Therefore the natural map ${{\rm Coker}(\Fr_*-1)\to{\Hom\big({\rm Ker}(\Fr_*-1),k^*\big)}}$ is also an isomorphism.

Now observe that ${\rm Ker}(\Fr_*-1)=\Pic^0(X)_{q-1}$. Further, using the well-known surjectivity of the map $\Fr_*-1$ on $\Pic^0(\overline{X})$, one shows that there is an isomorphism
$$
\Pic^0(X)/(q-1)\stackrel{\sim}\longrightarrow {\rm Coker}(\Fr_*-1)\,,\qquad [m]\longmapsto \Fr_*(\widetilde m)-\widetilde m\,.
$$
Altogether this proves the lemma.
\end{proof}

The following relation between the Weil pairing and the tame symbol was first proved by Howe~\cite{How}, then by Mazo~\cite{Maz} in a more elementary way, and then in~\cite[Cor.\,4.1]{Gor} by a different method based on the relation between the Poincar\'e biextension and the tame symbol mentioned in the proof of Theorem~\ref{theo:infinite}(i).

\begin{prop}\label{prop:Weiltame}
Let $\varphi\in K^*$ be such that ${\rm div}(\varphi)\in (q-1)\Div(X)$ and let $h\in \Ab^*_X$ be such that $h^{q-1}\in K^*$. Then there is an equality for the Weil pairing (see~\eqref{eq:divadeles} for the definition of~${\rm div}$)
$$
\big([{\rm div}(\varphi)/(q-1)],[{\rm div}(h)]\big)_{q-1}=(\varphi,h)_X\,.
$$
\end{prop}

Actually, Proposition~\ref{prop:Weiltame} holds for a smooth projective curve over an arbitrary field and for $n$-torsion in the Picard group, where $n$ is prime to the characteristic of the ground field.

\begin{cor}\label{cor:Weil}
For any $\varphi\in K^*$ such that ${\rm div}(\varphi)\in (q-1)\Div(X)$ and any $g\in (\Ab^*_X)^0$, there is an equality
$$
\kappa\big([{\rm div}(\varphi)/(q-1)],[{\rm div}(g)]\big)=(\varphi,g)_X\,.
$$
\end{cor}
\begin{proof}
We will use the ring of adeles $\Ab^*_{\bar X}$ of the curve $\overline{X}$ over the algebraic closure $\bar k$ of~$k$. Note that there is a natural embedding of rings $\Ab_X\subset \Ab_{\bar X}$.

It follows from $(q-1)$-divisibility of the groups $\Pic^0(\overline{X})$ and $\bar k[[t]]^*$ that there are ${\widetilde{g}\in \Ab^*_{\bar X}}$ and ${\psi\in\bar k(\overline{X})^*}$ such that ${\widetilde{g}^{\,q-1}=g\psi}$. In particular, there is an equality
\begin{equation}\label{eq:1metro}
(q-1)[{\rm div}(\widetilde{g})]=[{\rm div}(g)]
\end{equation}
in $\Pic^0(\overline{X})$.

The finite morphism $\Fr\colon \overline X\to \overline X$ defines the embedding $\Ab^*_{\bar X}\hookrightarrow \Ab^*_{\bar X}$ and the norm map $\Nm_{\,\Fr}\colon \Ab^*_{\bar X}\to \Ab^*_{\bar X}$. Put
$$
h=\Nm_{\,\Fr}(\widetilde g)\cdot\widetilde g^{\,-1}\in\Ab^*_{\bar X}\,.
$$
Then we have
\begin{equation}\label{eq:2metro}
h^{q-1}=\Nm_{\,\Fr}(\widetilde g^{\,q-1})\cdot \widetilde g^{\,-(q-1)}=\Nm_{\,\Fr}(g\psi)\cdot (g\psi)^{-1}=\Nm_{\,\Fr}(\psi)\cdot\psi^{-1}\in \bar k(\overline{X})^*\,,
\end{equation}
where the third equality follows from the fact that the restriction of $\Nm_{\,\Fr}$ to the subgroup $\Ab^*_X\subset \Ab^*_{\bar X}$ is the identity.

There is a commutative diagram
$$
\begin{CD}
 \Ab^*_{\bar X} @>{\rm [div(-)]}>> \Pic(\overline{X}) \\
 @V_{\Nm_{\,\Fr}}VV @V_{\Fr_*}VV \\
 \Ab^*_{\bar X} @>{\rm [div(-)]}>> \Pic(\overline{X})\,.
 \end{CD}
$$
Hence, we have
\begin{equation}\label{eq:1.5metro}
[{\rm div}(h)]=\Fr_*[{\rm div}(\widetilde{g})]-[{\rm div}(\widetilde{g})]\,.
\end{equation}

Combining formulas~\eqref{eq:1metro} and~\eqref{eq:1.5metro}, we obtain the equality
$$
\kappa\big([{\rm div}(\varphi)/(q-1)],[{\rm div}(g)]\big)=\big([{\rm div}(\varphi)/(q-1)],[{\rm div}(h)]\big)_{q-1}\,.
$$
By Proposition~\ref{prop:Weiltame} and formula~\eqref{eq:2metro}, we get
$$
\big([{\rm div}(\varphi)/(q-1)],[{\rm div}(h)]\big)_{q-1}=(\varphi,h)_{\bar X}\,.
$$
Finally, there are equalities
$$
(\varphi,h)_{\bar X}=\big(\varphi,\Nm_{\,\Fr}(\widetilde g)\big)\cdot(\varphi,\widetilde g^{\,-1})_{\bar X}=\Nm_{\,\Fr}(\varphi,\widetilde{g})_{\bar X}\cdot (\varphi,\widetilde{g})_{\bar X}^{-1}=
$$
$$
=(\varphi,\widetilde{g})^q_{\bar X}\cdot (\varphi,\widetilde{g})^{-1}_{\bar X}=(\varphi,\widetilde{g}^{\,q-1})_{\bar X}=(\varphi,g\psi)_{\bar X}=(\varphi,g)_{\bar X}=(\varphi,g)_X\,,
$$
where the second equality follows from the projection formula for the tame symbol pairing and the third equality follows from the fact that $\Nm_{\,\Fr}$ sends an element $c\in \bar k^*$ to $c^q$.
\end{proof}

\begin{lemma}\label{lem:Ufinite}
There are equalities
$$
U_x=(K_x^*)^{q-1}\,,\qquad U=(\Ab^*_X)^{q-1}\,,
$$
where $x\in X$ is any closed point, and the global tame symbol pairing induces an isomorphism $\Ab_X^*/U\simeq \Hom_c\big(\Ab_X^*/U,k^*\big)$, where $\Hom_c$ denotes the group of continuous homomorphisms.
\end{lemma}
\begin{proof}
The formula for $U_x$ follows from the explicit description of the local tame symbol pairing and from the facts for any finite extension of fields $k\subset l$, the multiplicative group $1+tl[[t]]$ is $(q-1)$-divisible and the norm map $\Nm_{l/k}$ gives an isomorphism ${l^*/(l^*)^{q-1}\simeq k^*}$. The formula for $U$ follows from Remark~\ref{rmk:kerneltame}.

We see that there is an exact sequence
$$
1\longrightarrow \mbox{$\prod\limits_{x\in X}k^*$}\longrightarrow \Ab^*_X/(\Ab^*_X)^{q-1}\longrightarrow \Div(X)/(q-1)\longrightarrow 0
$$
and the tame symbol pairing induces isomorphisms
$$
\mbox{$\prod\limits_{x\in X}k^*$}\stackrel{\sim}\longrightarrow \Hom\big(\Div(X)/(q-1),k^*\big)\,,\qquad
\Div(X)/(q-1)\stackrel{\sim}\longrightarrow \Hom_c\big(\mbox{$\prod\limits_{x\in X}k^*$},k^*\big)\,.
$$
This proves the lemma.
\end{proof}

\begin{theo}\label{theo:finite}
Assume that $k=\FF_q$ is a finite field. Then there is an equality $K^{*\,\bot}=K^*\cdot U$.
\end{theo}
\begin{proof}
It follow from Lemma~\ref{lem:Ufinite} and its proof that there are equalities
$$
C=\mbox{$\prod\limits_{x\in X}k^*$}\,,\qquad A/C=\Div(X)/(q-1)
$$
and the map $\alpha$ coincides with the natural isomorphism
$$
\mbox{$\prod\limits_{x\in X}k^*$}\stackrel{\sim}\longrightarrow \Hom\big(\Div(X)/(q-1),k^*\big)\,.
$$
Let $F\subset K^*$ be the subgroup that consists of all $\varphi\in K^*$ such that ${{\rm div}(\varphi)\in(q-1)\Div(X)}$. Clearly, there are embeddings $k^*,(K^*)^{q-1}\subset F$. It follows from Corollary~\ref{cor:AT}(i) and Lemma~\ref{lem:Ufinite} that there is an equality ${B\cap C=F/(K^*)^{q-1}}$. Also, there is an exact sequence
\begin{equation}\label{eq:F}
1\longrightarrow k^*\longrightarrow F/(K^*)^{q-1}\longrightarrow \Pic^0(X)_{q-1}\longrightarrow 0\,,
\end{equation}
where the second map sends the class $[\varphi]$ of $\varphi\in F$ to $[{\rm div}(\varphi)/(q-1)]$.

On the other hand, we have $A/(B+C)=\Pic(X)/(q-1)$ and by Corollary~\ref{cor:AT}(ii), there is an exact sequence
\begin{equation}\label{eq:Pic}
0\longrightarrow\Pic^0(X)/(q-1)\longrightarrow\Pic(X)/(q-1)\stackrel{\deg}\longrightarrow \Z/(q-1)\longrightarrow 0\,.
\end{equation}
The global tame symbol pairing induces naturally the pairing between the groups ${B\cap C=F/(K^*)^{q-1}}$ and $A/(B+C)=\Pic(X)/(q-1)$. Using exact sequences~\eqref{eq:F},~\eqref{eq:Pic} and Corollary~\ref{cor:Weil}, we see that this induces pairings
$$
k^*\times \Z/(q-1)\longrightarrow k^*\,,\qquad (c,[n])\longmapsto c^n\,,
$$
$$
\kappa\;:\;\Pic^0(X)_{q-1}\times \Pic(X)/(q-1)\longrightarrow k^*\,.
$$
The first pairing is obviously unimodular, while the second one is unimodular by Lemma~\ref{lemma:nondegen}. It follows that condition~(ii) of Corollary~\ref{cor:key} is satisfied, which finishes the proof.
\end{proof}

Let ${C_X=\Ab_X^*/K^*}$ be the idele class group of $X$ and let $G_K^{\rm ab}$ be the universal abelian quotient of the Galois group $G_K$ of the field~$K$.

\begin{cor}\label{cor:CFT}
The global tame symbol pairing
$$
(\cdot,\cdot)_X\;:\;K^*\times C_X\longrightarrow k^*
$$
induces an isomorphism of topological groups
$$
C_X/(q-1)\stackrel{\sim}\longrightarrow \Hom\big(K^*/(K^*)^{q-1},k^*\big)\,,
$$
which, followed by the isomorphism from the Kummer theory, provides an isomorphism of topological groups
$$
C_X/(q-1)\stackrel{\sim}\longrightarrow G_K^{\rm ab}/(q-1)\,.
$$
\end{cor}
\begin{proof}
By Lemma~\ref{lem:Ufinite}, the global tame symbol pairing induces an isomorphism ${K^{*\,\bot}/(K^{*\,\bot}\cap U)\simeq \Hom_c\big(C_X/(q-1),k^*\big)}$. By Theorem~\ref{theo:finite} and Corollary~\ref{cor:AT}(i), we have ${K^{*\,\bot}/(K^{*\,\bot}\cap U)=K^*/(K^*)^{q-1}}$. Since $k^*$ is a cyclic group of order $q-1$, we conclude that the $(q-1)$-torsion groups $K^*/(K^*)^{q-1}$ and $C_X/(q-1)$ are Pontryagin dual with respect to the global tame symbol pairing.
\end{proof}

\end{document}